\documentclass[12pt,a4paper]{amsart}
\usepackage{amssymb,latexsym,amsmath,amscd,graphicx,url,color,comment}
\numberwithin{equation}{section}
\let\oldlabel=\label
\def\prellabel{\marginparsep=1em\marginparwidth=44pt
    \def\label##1{\oldlabel{##1}\ifmmode\else\ifinner\else
         \marginpar{{\footnotesize\ \\ \tt
                    ##1}}\fi\fi}}

\theoremstyle{plain}
\newtheorem{thm}{Theorem}[section]
\newtheorem{prop}[thm]{Proposition}
\newtheorem{cor}[thm]{Corollary}
\newtheorem{lemma}[thm]{Lemma}

\theoremstyle{definition}
\newtheorem{defn}[thm]{Definition}
\newtheorem{rmk}[thm]{Remark}

\newtheorem{quest}[thm]{Question}
\newtheorem{ex}[thm]{Example}

\newcommand{\NN}{{\mathbb N}}
\newcommand{\PP}{{\mathbb P}}
\newcommand{\QQ}{{\mathbb Q}}

\newcommand{\ZZ}{{\mathbb Z}}

\newcommand{\FF}{{\mathbb F}}

\DeclareMathOperator{\Paths}{Paths}

\DeclareMathOperator{\CS}{CS}
\DeclareMathOperator{\Deg}{Deg}
\DeclareMathOperator{\HS}{HS}
\DeclareMathOperator{\HF}{HF}
\DeclareMathOperator{\GL}{GL}
\DeclareMathOperator{\GCD}{GCD}
\DeclareMathOperator{\LCM}{LCM}
\DeclareMathOperator{\gin}{gin}
\DeclareMathOperator{\inid}{in}

\DeclareMathOperator{\reg}{reg}

\DeclareMathOperator{\proj}{Proj}
\DeclareMathOperator{\ess}{Ess}

\title{Radical generic initial ideals}

\author{A. Conca}
\address{Dipartimento di Matematica, 
Universit\`a di Genova, Via Dodecaneso 35, 
I-16146 Genova, Italy}
\email{conca@dima.unige.it}

\author{E. De Negri}
\address{Dipartimento di Matematica, 
Universit\`a di Genova, Via Dodecaneso 35, 
I-16146 Genova, Italy}
\email{denegri@dima.unige.it}

\author{E. Gorla}
\address{Institut de Math\'ematiques, Universit\'e de Neuch\^atel, Rue Emile-Argand 11, CH-2000
  Neuch\^atel, Switzerland}  
\email{elisa.gorla@unine.ch}

\thanks{}


\begin{document}

\begin{abstract}
In this paper, we survey the theory of Cartwright-Sturmfels ideals. These are $\ZZ^n$-graded ideals, whose multigraded generic initial ideal is radical. Cartwright-Sturmfels ideals have surprising properties, mostly stemming from the fact that their Hilbert scheme only contains one Borel-fixed point. This has consequences, e.g., on their universal Gr\"obner bases and on the family of their initial ideals. In this paper, we discuss several known classes of Cartwright-Sturmfels ideals and we find a new one. Among determinantal ideals of same-size minors of a matrix of variables and Schubert determinantal ideals, we are able to characterize those that are Cartwright-Sturmfels.
\end{abstract}

\maketitle

\section*{Introduction}
In 2010 Cartwright and Sturmfels published a paper \cite{CS} containing surprising results on certain multigraded ideals. More precisely, they proved that any $\ZZ^n$-multigraded ideal that has the $\ZZ^n$-multigraded Hilbert function of the ideal of $2$-minors of an $m\times n$ generic matrix must be radical and Cohen-Macaulay. During our stay at MSRI in 2012, we realised  that a similar phenomenon was related to the  universal Gr\"obner basis theorem for maximal minors proved in the nineties  by Bernstein, Sturmfels, and Zelevisky \cite{SZ,BZ}. We managed to identify a  notion that ``explains"  the common features behind  these two settings and that is flexible enough to be useful in other contexts. The key idea is to consider the family of multigraded ideals with radical multigraded generic ideals, that we named after Cartwright and Sturmfels. We wrote four papers related to the subject \cite{CDG1,CDG3,CDG4,CDG2}. The goal of this note is to give a short introduction to Cartwright-Sturmfels ideals, to highlight their properties, and to present some classes of Cartwright-Sturmfels ideals, both old and new. In particular,  in Section \ref{DetCS} we classify determinantal ideals that are Cartwright-Sturmfels in the generic case and derive results for the non-generic case. In Section \ref{sect:schubert} we characterize Schubert determinantal ideals that are Cartwright-Sturmfels. 
In Section \ref{BinCS}  we take the occasion to correct a mistake in the proof of Theorem 2.1  of \cite{CDG3} asserting that any binomial edge ideals is Cartwright-Sturmfels. Finally, in Section \ref{ClosureCS} we recall another result  from \cite{CDG3} asserting that the multiprojective closure of any linear ideal is Cartwright-Sturmfels and conclude with a question suggested by it.  

{\bf Acknowledgements:} The authors thank Anna Weigandt for useful discussions on the material of this paper.

The first and the second authors were partially supported by GNSAGA-INdAM.

\section{Multigraded generic initial ideals and Multidegree}\label{multidegreeStructure}
Let  $n\in\NN_+$ and $m_1,\ldots,m_n\in\NN$. 
Let $S=K[x_{ij}\mid 1\leq j\leq n,\ 0\leq i\leq m_j]$ be a polynomial ring over a field $K$ endowed with the standard $\ZZ^n$-grading induced by setting $\deg(x_{ij})=e_j$, where $e_j\in\ZZ^n$ is the $j$-th standard basis vector. 

We will deal with  $\ZZ^n$-graded ideals and modules of $S$. We use the words $\ZZ^n$-graded and multigraded interchangeably. 
For simplicity we always assume the term orders on $S$ satisfy $x_{0j}>x_{1j}>\ldots>x_{m_jj}$ for all $j=1,\ldots,n$.

The ring $S$ may be thought of as 
the coordinate ring of the product of $n$ projective spaces, i.e. 
$$\proj(S)=\PP^{(m_1,\dots,m_n)}=\PP^{m_1}\times\ldots\times\PP^{m_n}.$$
A  multigraded prime ideal $P$  of $S$ is relevant if $P$ does not contain $S_{(1,1,\dots,1)}$ and irrelevant otherwise. When $K$ is algebraically closed relevant prime ideals correspond to irreducible subvarieties of $\PP^{(m_1,\dots,m_n)}$. 
\subsection{The multigin}  
The group $G=\GL_{m_1+1}(K)\times\cdots\times\GL_{m_n+1}(K)$ acts naturally on $S$ as the group of multigraded $K$-algebra automorphisms,  i.e., coordinate changes that fix each factor in the product of projective spaces. Let $I$ be   a multigraded ideal of $S$ and let $\sigma$ be a term order  on $S$.   As in  the standard $\ZZ$-graded situation, if $K$ is infinite  there exists a nonempty Zariski open $U\subseteq G$ such that $\inid_\sigma(gI)=\inid_\sigma(g^\prime I)$ for all $g,g^\prime\in U$. This leads  to the definition of multigraded generic initial ideal. We refer the   reader to~\cite[Theorem~15.23]{E} for details on the generic initial ideals in the $\ZZ$-graded case and to~\cite[Section~1]{ACD} for a similar discussion in the $\ZZ^n$-graded case.

\begin{defn}
The  multigraded generic initial ideal $\gin_\sigma(I)$ of $I$ with respect to $\sigma$ is the ideal 
$\inid_\sigma(gI)$, where $g$ is a generic multigraded coordinate change, i.e. $g\in U$ and $U$ is a nonempty Zariski open subset of $G$. 
\end{defn}

Let $B=B_{m_1+1}(K)\times \cdots  \times B_{m_n+1}(K)$ be the {\em Borel
  subgroup} of $G$, consisting of the upper triangular invertible
matrices in $G$. One knows  that $\gin_\sigma(I)$ is Borel fixed, that is, it is fixed  by the action of every $g\in B$.

\subsection{Multidegree and dual multidegree}  

For a finitely generated $\ZZ^n$-graded module $M=\oplus_{a\in \ZZ^n} M_a$ over a standard $\ZZ^n$-graded polynomial ring $S$, one may define the multigraded Hilbert function as the function $\HF(M,-)$ that associates to $a=(a_1,\dots, a_n) \in \ZZ^n$ the number  $\HF(M,a)=\dim_K M_a$. As in the $\ZZ$-graded case, for  $a\gg 0$ the multigraded Hilbert function  agrees with a polynomial in $n$ variables $P_M(Z)=P_M(Z_1,\dots, Z_n)$,  the multigraded Hilbert polynomial of $M$. 

Let $d(M)$ be the total degree of $P_M(Z)$.  Under mild assumptions, for example when  all the minimal primes of $M$ are relevant,  one has  that $d(M)=\dim(M)-n$. The homogeneous component of degree $d(M)$ of $P_M(Z)$  can be written as 
$$\sum_{} \frac{e_{b}(M)}{b_1!b_2!\cdots b_n!} Z_1^{b_1}\cdots Z_n^{b_n},$$
where the sum ranges over all $b\in \NN^n$ with $b_i\leq m_i$ such that $|b|=d(M)$. The numbers $e_{b}(M)$ are the multidegrees (or mixed multiplicities) of $M$. It turns out that they are non-negative integers. The multidegree of $M$ is the polynomial
$$\Deg_M(Z_1,\dots, Z_n)=\sum_{}  e_{b}(M)   Z_1^{b_1}\cdots Z_n^{b_n},$$
where the sum is over all $b\in \NN^n$ such that $|b|=d(M)$.  

One can regard $M$ as a $\ZZ$-graded module by $M_v=\oplus_{|a|=v} M_a$. With respect to this  $\ZZ$-grading, $M$ has an ordinary multiplicity $e(M)$ and, if  all the minimal primes of $M$ are relevant, one has 
\begin{equation}\label{sum-mm}
e(M)=\sum_{b} e_b(M)
\end{equation}
where the sum ranges over all the $b\in \NN^n$ such that $|b|=\dim(M)-n$. This is proved in \cite[Theorem 2.8]{CR}, but a special case appears already in \cite{vW2}. 

When $M$ is the coordinate ring of an irreducible  multiprojective variety $X \subseteq \PP^{(m_1,\dots,m_n)}$ over an algebraically closed field, the multidegrees $e_{b}(M)$'s  have a geometric interpretation. Indeed, in that case $e_{b}(M)$ is the number of points of $\PP^{(m_1,\dots,m_n)}$ that one gets by intersecting $X$ with $L_1\times L_2 \times \dots \times L_n$ where each $L_i$ is a general linear subspace of $\PP^{m_i}$ of codimension $b_i$.  

\begin{defn}   
A $\ZZ^n$-graded module $M$ has a multiplicity-free multidegree if $e_b(M)\in \{0, 1\}$  for all $b$  with $|b| = d(M)$. 
\end{defn} 

The relevant prime ideals $P$ of $S$ such that $S/P$ have a multiplicity-free multidegree are studied in \cite{B} by Brion (in a more general setting) who proves in particular that $S/P$ is Cohen-Macaulay.  

The word multidegree is used in the literature also to refer to another polynomial invariant of $M$, which we call dual multidegree throughout this paper, in order to avoid confusion. The dual multidegree is defined as follows:
The multigraded Hilbert series of $M$ is
$$\HS(M,Z)=\sum_{a\in \ZZ^n} (\dim_K M_a)  Z^a\in\QQ[[Z_1,\dots,Z_n]][Z_1^{-1},\dots, Z_n^{-1}].$$ 
Let $$K_M(Z)= \HS(M,Z)\prod_{i=1}^n (1-Z_i)^{m_i+1}.$$
It turns out that $K_M(Z) \in \ZZ[Z_1^{\pm 1} ,\dots,Z_n^{\pm 1}]$. The dual multidegree $\Deg^*_M(Z)$ of $M$ is the homogeneous component of smallest total degree of $K_M(1-Z_1,\dots,1-Z_n)$. One can show that $\Deg^*_M(Z)\in \NN[Z_1,\dots,Z_n]$. Notice that the dual multidegree corresponds to the multidegree as defined in e.g. \cite{KM,MS}.

\subsection{Multidegrees of  radical monomial ideals} 

Let $J$ be  a  radical  monomial ideal of $S$ with associated simplicial complex 
$\Delta \subseteq  2^{T}$. Here $T=\{ (i,j) :  1\leq j\leq n,\  0\leq i\leq m_j \}$. 
The ideal $J$ is naturally $\ZZ^{|T|}$-graded, however here we consider its $\ZZ^n$-graded structure and describe its multigraded Hilbert polynomial in terms of $\Delta$. For each $F\in \Delta$ and $j\in [n]$ we set  
$$c_j(F)=|\{ (0,j), \dots, (m_j,j)\} \cap F|$$  
 and  $c(F)=(c_1(F),\dots, c_n(F))\in \NN^n$. A face $F$ is relevant if the corresponding prime ideal $(x_{ij}  : (i,j) \not\in F)$ is relevant,  i.e., $c_j(F)>0$ for every $j=1,\dots, n$. Let us denote by $R(\Delta)$ the set of the relevant faces of $\Delta$, i.e., 
 $$R(\Delta)=\{ F \in \Delta : c_j(F)>0 \mbox{ for all }  j\in [n]\}.$$
 
\begin{lemma} 
\label{multiHF-easy} 
For every $a=(a_1,\dots, a_n)\in \NN_+^n$ one has 
$$\HF(S/J,a)=\sum_{F\in R(\Delta)} \prod_{j=1}^n\binom{a_j-1}{c_j(F)-1},$$
in particular
$$P_{S/J}(Z_1,\dots, Z_n)=\sum_{F\in R(\Delta) }  \prod_{j=1}^n \binom{Z_j-1}{c_j(F)-1}.$$ 
\end{lemma} 

\begin{proof} 
First we observe that $\HF(S/J,a)$ is the number of monomials in $S$ of multidegree $a$ which are not contained in $J$. To a monomial  $x^v=\prod x_{ij}^{v_{ij}}$ we may associate its support $F(x^v)=\{ (i,j)  : v_{ij}>0\}$. Since $a_i>0$ for all $i$, by construction we have $x^v\not\in J$ and $\deg(x^v)=a$ if and only if $F(x^v)\in R(\Delta)$.  We partition the set of monomials of degree $a$ which do not belong $J$ according to their support.  
The monomials of degree $a$ supported on a given $F\in R(\Delta)$ have the form $(\prod_{(i,j)\in F} x_{ij} )x^v$, where $x^v$ is a monomial with support contained in $F$ and degree $a-c(F)$. The number of these monomials is $\prod_{j=1}^n \binom{a_j-1}{c_j(F)-1}$. 
\end{proof} 
   
Denote by $\FF(\Delta)$ the set of the facets of $\Delta$.  Recall that  $\Delta$ is a pure simplicial complex  if all the  facets of $\Delta$ have the same dimension. As an immediate corollary we have
   
\begin{lemma}\label{multig-easy1}  
Assume that $\Delta$ is a pure simplicial complex and that $\FF(\Delta) \cap  R(\Delta) \neq  \emptyset$. Then 
$$\Deg_{S/J}(Z_1,\dots, Z_n)=\sum_{F\in \FF(\Delta) \cap  R(\Delta)} Z_1^{c_1(F)-1}\cdots Z_n^{c_n(F)-1}.$$
\end{lemma} 


\section{Cartwright-Sturmfels ideals}\label{SectCS}

In this section we recall the definition of Cartwright-Sturmfels ideals and some facts about them, which were discussed in our papers \cite{CDG1,CDG3,CDG4,CDG2}. 

\begin{defn}
A multigraded ideal $I$ of $S$ is a Cartwright-Sturmfels ideal if it has a radical multigraded generic initial ideal.
\end{defn}

We denote by $\CS(S)$, or simply by $\CS$ if $S$ is clear from the context, the family of Cartwright-Sturmfels ideals of $S$. 

\begin{ex}\label{ex1} 
The $\ZZ$-graded Cartwright-Sturmfels ideals are exactly those generated by linear forms. In fact, if $I$ is not generated by   linear forms, let $d>1$ be the least degree of a minimal generator of $I$ which is not linear. Then the generic initial ideal of $I$ has a minimal generator which is the $d$-th power of a variable. In particular, the generic initial ideal of $I$ with respect to any term order is not radical. 
\end{ex}

Notice that the property of being Cartwright-Sturmfels depends on the multigrading.

\begin{ex}
If $I\subseteq S$ is  generated by  a non-zero element of degree $(1,1,\dots, 1)\in \ZZ^n$ then $I$ is a $\ZZ^n$-graded Cartwright-Sturmfels ideal for obvious reasons. However, as we have observed in Example \ref{ex1}, the ideal $I$ is not a $\ZZ$-graded Cartwright-Sturmfels ideal if $n>1$. 
\end{ex}

Cartwright-Sturmfels ideals have many interesting properties. The next proposition summarizes some of them. 




\begin{prop}\label{oldprop}
Let $I\in \CS$ and let $J$ be a radical Borel fixed ideal such that $\HF(I,a)=\HF(J,a)$ for all $a\in \NN^n$. Then:
\begin{enumerate}
\item  $I$ is radical and $\gin_\tau(I)=J$ for every term order $\tau$. 
\item  $\inid_\tau(I)\in \CS$, in particular it is square free, for every term order $\tau$. 
 \item $\reg(I)\leq n$.
 \item If $K$ is algebraically closed, then $P\in \CS$ for every minimal prime $P$ of $I$. 
\item  $I$ is generated by elements of multidegree $\leq (1,\ldots, 1)$.
\item All reduced Gr\"obner bases of $I$ consist of elements of multidegree $\leq (1,\ldots, 1)$.\\ In particular, $I$ has a universal  Gr\"obner basis of elements of multidegree $\leq (1,\ldots, 1)$. 
\end{enumerate}
\end{prop} 

Moreover one has

\begin{prop}\label{oldprop1}
Let $I$ be a multigraded ideal of $S$. Then $I\in\CS$ if and only if there exists a radical Borel-fixed multigraded ideal which has the same multigraded Hilbert series as $I$.
\end{prop} 

The family $\CS$ is closed under some natural operations.

\begin{prop}\label{closures}
Let $L$ be a $\ZZ^n$-graded  linear form of $S$. In the following  $S/(L)$ is identified with a polynomial ring with the induced $\ZZ^n$-graded structure.
Let $U_i\subseteq S_{e_i}$ be vector subspaces for all $i=1,\ldots,n$ and let $R=K[U_1,\ldots,U_n]$ be the $\ZZ^n$-graded polynomial subring of $S$ that they generate. 
Then:
\begin{enumerate}
\item If $I\in \CS(S)$, then $I:L\in \CS(S)$. 
\item If $I\in \CS(S)$, then  $I+(L)\in \CS(S)$ and $I+(L)/(L)\in \CS(S/(L))$.
\item If $I\in \CS(S)$, then $I\cap R\in \CS(R)$. 
\end{enumerate}
\end{prop}

Moreover one has

\begin{prop}\label{oldprop2} 
\begin{itemize} 
\item [(1)] If $I\in \CS(S)$ then $S/I$ has multiplicity-free multidegree. 
\item[(2)] Vice versa, suppose that $K$ is algebraically closed, $I$ is a relevant prime ideal of $S$, and $S/I$ has multiplicity-free multidegree. Then $I\in \CS(S)$.
\end{itemize}
\end{prop}

Part (2) is proved in \cite{B} using a different terminology. 


\section{Determinantal Cartwright-Sturmfels ideals}
\label{DetCS}

The goal of this section is to discuss Cartwright-Sturmfels ideals that are generated by minors of matrices. New results on the family of Schubert determinantal ideals will be presented in Section \ref{sect:schubert}.

We start by discussing generic determinantal ideals, i.e., ideals of same-size minors of the  matrix of variables. 
Let $X=(x_{ij})$ be an $m\times n$ matrix of variables over a field $K$ and $S=K[x_{ij} : 1\leq j\leq n \mbox{ and } 1\leq i \leq m]$. We consider the $\ZZ^n$-graded structure on $S$ induced by $\deg(x_{ij})=e_j\in \ZZ^n$. In the notation of Section \ref{multidegreeStructure} we have $m_j=m-1$ for $j=1,\dots, n$ and, in accordance with usual notation for matrices, the index $i$ varies from $1$ to $m$. 
Let $I_t(X)$ be the ideal of $S$ generated by the $t$-minors of $X$. Clearly  $I_t(X)$ is $\ZZ^n$-graded and our first goal is to compute the multidegree of $S/I_t(X)$. 

The multigraded Hilbert function, hence the multidegree, does not change if we replace $I_t(X)$ with an initial ideal. The ideal $I_t(X)$ has a well-known square free initial ideal, discussed in \cite{S, HT, BC}. It is the ideal generated by the products of the entries on the main diagonals of the $t$-minors, whose associated simplicial complex will be denoted by $\Pi_t$. The facets of $\Pi_t$ can be identified with the non-intersecting paths in the grid $[m]\times [n]$ from the starting points $p_1=(1,n), p_2=(2,n),\ldots, p_{t-1}=(t-1,n)$ to the endpoints $q_1=(m,1), q_2=(m,2),\ldots, q_{t-1}=(m, t-1)$. 

For example, for $m=4$, $n=5$, and $t=3$, we have $p_1=(1,5)$, $p_2=(2,5)$, $q_1=(4,1)$, and $q_2=(4,2)$. 
The following is a facet of $\Pi_3$:
$$\begin{array}{ccccccc}
- & 1 &  1  &  1  &1  & \leftarrow  p_1  \\
1 & 1 & -  &  2   &2  & \leftarrow  p_2  \\
1 & - & 2 &  2  & - \\
 1& 2  & 2  &  -  & - \\
 \uparrow & \uparrow \\
 q_1  & q_2 \\
\end{array}$$
depicted using the matrix coordinates and marking with ``1" the lattice points which belong to the first path (from $p_1$ to $q_1$) and with ``2" the lattice points of the second path (from $p_2$ to $q_2$). 

Each family of non-intersecting paths must have at least $t-1$ points on each column. Therefore, each facet of $\Pi_t$ is relevant if $t>1$. With the notation of Section \ref{multidegreeStructure}, $\FF(\Pi_t)\subseteq R(\Pi_t)$. Summing up, by Lemma \ref{multig-easy1} we have
 
\begin{equation}\label{MD1} 
\Deg_{S/I_t(X)}(Z_1,\dots,Z_n)=\sum_{F\in \FF(\Pi_t)}  Z_1^{c_1(F)-1}\cdots Z_n^{c_n(F)-1},
\end{equation} 
where $c_j(F)=|\{ (a,b)\in F : b=j\}|$.
 
We introduce the generating function associated to the statistics $c(F)$. Given a collection $U$ of subsets of $[m]\times [n]$ we set 
$$W(U,Z_1,\dots, Z_n)=\sum_{F\in U} Z_1^{c_1(F)}\cdots Z_n^{c_n(F)},$$  
 so that we may rewrite (\ref{MD1}) as
 
\begin{equation}\label{MD2} 
\Deg_{S/I_t(X)}(Z_1,\dots,Z_n) =(Z_1\cdots Z_n)^{-1} W(\FF(\Pi_t),Z_1,\dots, Z_n).
\end{equation} 
 
Next we give a determinantal formula for $W(\FF(\Pi_t),Z_1,\dots, Z_n)$. One observes that the Gessel--Viennot involution \cite{GV}, used in \cite{HT} to compute $|\FF(\pi_t)|$, is compatible with any weight given to the lattice points. Hence one gets immediately

\begin{equation}\label{MD3} 
W(\FF(\Pi_t),Z_1,\dots, Z_n)=\det \left( W( \Paths(p_i,q_j) ,Z_1,\dots, Z_n)\right)_{i,j=1,\dots,  t-1}
\end{equation} 
where  $\Paths(p_i,q_j)$ is the set of the paths for $p_i$ to $q_j$.  

In the sequel, $h_v(L)$ denotes the complete homogeneous symmetric polynomial of degree $v$ on the set $L$, i.e., the sum of all monomials of degree $v$ in the elements of $L$.  

\begin{lemma}\label{weightonepath} 
Given $p=(a,b)$ and $q=(c,d)$ with $a \leq c$ and $b\geq d$, we have 
$$W(\Paths(p,q) ,Z_1,\dots, Z_n)=\left(\prod_{i=d}^b Z_i \right) h_{c-a}(Z_d,Z_{d+1},\dots, Z_b).$$ 
\end{lemma} 

\begin{proof} 
Any path $P$ from $p$ to $q$ is uniquely determined by the the number of points of intersection with the columns. Any such path must have at least one point on column $j$ for all $d\leq j\leq b$, and no points on the other columns. The only other constraint is that the path has $(b-d)+(c-a)+1$ points in total. In terms of $c(P)=(c_1(P),\dots, c_n(P))$ the constrains are $c_j(P)>0$ if and only if $d\leq j\leq b$ and $\sum_{j=d}^b c_j(P)=(b-d)+(c-a)+1$. Expressing this in terms of generating functions yields the desired result.
\end{proof} 

We can now compute the multidegree of generic determinantal rings.

\begin{thm}\label{multidegdet}
The multidegree of the determinantal ring $S/I_t(X)$ of an $m\times n$ matrix of variables $X$  with  $2\leq t\leq \min\{m,n\}$ and with respect to the $\ZZ^n$-graded structure induced by $\deg(x_{ij})=e_j$ is 
$$\Deg_{S/I_t(X)}(Z_1,\dots,Z_n)=(Z_1\cdots Z_n)^{t-2} 
\det\left(h_{m+1-t-i+j}(Z_1,\dots,Z_n)\right)_{i,j=1,2,\dots, t-1}.$$
\end{thm}

\begin{proof} 
Combining (\ref{MD2}) and (\ref{MD3}) with Lemma \ref{weightonepath} we obtain
$$\Deg_{S/I_t(X)}(Z_1,\dots,Z_n) = (Z_1\cdots Z_n)^{-1} \det\left(\left(\prod_{k=j}^n Z_k\right) 
h_{m-i}(Z_j,\dots, Z_n)\right)_{i,j=1,2,\dots, t-1}.$$ 
For $j=1,\dots,t-1$, the factor $\prod_{k=j}^n Z_k $ can be extracted from the determinant, hence the monomial in front of the determinant becomes
$$Z_2Z_3^{2}\cdots Z_{t-2}^{t-3}  Z_{t-1}^{t-2}  Z_{t}^{t-2} \cdots Z_{n}^{t-2},$$
while the determinant becomes
$$\det\left(h_{m-i}(Z_j,\dots, Z_n)\right)_{i,j=1,2,\dots, t-1}.$$ 
It now suffices to prove that the latter equals
$$Z_1^{t-2}Z_2^{t-3}\cdots Z_{t-2} \det\left( h_{m+1-t-i+j}(Z_1,\dots, Z_n)\right).$$

Let us explain this last equality in full detail in the case $t=4$,
which is general enough to show all the relevant features. We wish to prove the equality
\begin{equation}\label{eqtoprove}
\begin{array}{r}
\det
\left(
\begin{array}{lll}
h_{m-1}(Z_1,\dots, Z_n)  & h_{m-1}(Z_2,\dots, Z_n) &  h_{m-1}(Z_3,\dots, Z_n)  \\
h_{m-2}(Z_1,\dots, Z_n)  & h_{m-2}(Z_2,\dots, Z_n) &  h_{m-2}(Z_3,\dots, Z_n)  \\
h_{m-3}(Z_1,\dots, Z_n)  & h_{m-3}(Z_2,\dots, Z_n) &  h_{m-3}(Z_3,\dots, Z_n)
\end{array}
\right)\phantom{.}
\\
= Z_1^2Z_2\det
\left(
\begin{array}{lll}
h_{m-3}(Z_1,\dots, Z_n)  & h_{m-2}(Z_1,\dots, Z_n)  &  h_{m-1}(Z_1,\dots, Z_n)    \\
h_{m-4}(Z_1,\dots, Z_n)  & h_{m-3}(Z_1,\dots, Z_n)  &  h_{m-2}(Z_1,\dots, Z_n)  \\
h_{m-5}(Z_1,\dots, Z_n)  & h_{m-4}(Z_1,\dots, Z_n)  &  h_{m-3}(Z_1,\dots, Z_n)
\end{array}
\right).
\end{array}\end{equation}

In the first term of (\ref{eqtoprove}), we subtract the second column from the third and the first column from the second. Since 
$h_{v}(Z_{j+1},\dots, Z_n)-h_{v}(Z_j,\dots, Z_n)=-Z_j h_{v-1}(Z_j,\dots, Z_n)$, we can factor out $-Z_2$ from the third column and $-Z_1$ from the second. The first term of (\ref{eqtoprove}) therefore becomes
$$Z_1Z_2 \det\left(
\begin{array}{lll}
h_{m-1}(Z_1,\dots, Z_n)  & h_{m-2}(Z_1,\dots, Z_n) &  h_{m-2}(Z_2,\dots, Z_n)  \\
h_{m-2}(Z_1,\dots, Z_n)  & h_{m-3}(Z_1,\dots, Z_n) &  h_{m-3}(Z_2,\dots, Z_n)  \\
h_{m-3}(Z_1,\dots, Z_n)  & h_{m-4}(Z_2,\dots, Z_n) &  h_{m-4}(Z_2,\dots, Z_n)
\end{array}\right).$$
Subtracting the second column from the third and factoring $-Z_1$, we obtain
$$-Z_1^2Z_2\det\left(\begin{array}{lll}
h_{m-1}(Z_1,\dots, Z_n)  & h_{m-2}(Z_1,\dots, Z_n) &  h_{m-3}(Z_1,\dots, Z_n)  \\
h_{m-2}(Z_1,\dots, Z_n)  & h_{m-3}(Z_1,\dots, Z_n) &  h_{m-4}(Z_1,\dots, Z_n)  \\
h_{m-3}(Z_1,\dots, Z_n)  & h_{m-4}(Z_2,\dots, Z_n) &  h_{m-5}(Z_1,\dots, Z_n)
\end{array}\right)$$
Finally we exchange rows one and three and then transpose. This yields 
$$Z_1^2Z_2\det\left(
\begin{array}{lll}
h_{m-3}(Z_1,\dots, Z_n)  & h_{m-2}(Z_1,\dots, Z_n)  &  h_{m-1}(Z_1,\dots, Z_n)    \\
h_{m-4}(Z_1,\dots, Z_n)  & h_{m-3}(Z_1,\dots, Z_n)  &  h_{m-2}(Z_1,\dots, Z_n)  \\
h_{m-5}(Z_1,\dots, Z_n)  & h_{m-4}(Z_1,\dots, Z_n)  &  h_{m-3}(Z_1,\dots, Z_n)
\end{array}\right)$$
which is the second term of (\ref{eqtoprove}).
\end{proof} 

The determinant  that appears in the statement of Theorem \ref{multidegdet} is a Schur polynomial. We refer to \cite{M} and \cite{Sta} for a treatment of the theory of symmetric functions and  Schur polynomials. Here we collect only the definitions and the properties that we will use in the sequel. 
A partition $\lambda$ is a weakly decreasing sequence of non-negative integers  $\lambda_1,\lambda_2,\ldots,\lambda_r$. Given a partition $\lambda=\lambda_1,\lambda_2,\ldots,\lambda_r$ the Schur polynomial $s_\lambda(Z)$ associated $\lambda$ and with respect to the variables 
$$Z=Z_1,\dots, Z_n$$
is  
$$s_\lambda(Z)=\det \left( h_{\lambda_i-i+j}(Z)  \right)_{i,j=1,\dots,r}.$$ 
By construction $s_\lambda(Z)$ is a symmetric homogeneous polynomial of degree $|\lambda|=\sum_{i=1}^r \lambda_i$ with integral coefficients.  Denote by $m_\mu(Z)$ the monomial symmetric polynomial associated with the partition $\mu=\mu_1\geq \mu_2 \geq \dots \geq \mu_n\geq 0$. Since the  $m_\mu(Z)$'s form a $K$-basis of the space of symmetric polynomials, one can express $s_\lambda(Z)$ as
$$s_\lambda(Z)=\sum_{\mu \ :\  |\mu|=|\lambda|} K_{\lambda, \mu} m_\mu(Z).$$
The   coefficients $K_{\lambda, \mu}$ are  known as the  Kostka numbers of the pair of partitions $\lambda, \mu$. We recall here their main properties. 
\begin{prop}
\label{Kostka numbers} 
For every pair of partitions $\lambda=\lambda_1,\lambda_2, \dots,  \lambda_r$ and $\mu=\mu_1, \mu_2,    \dots, \mu_n$ with $|\lambda|=|\mu|$ one has: 
\begin{itemize} 
\item[(1)] $K_{\lambda, \mu}\in \NN$. 
\item[(2)] $K_{\lambda, \mu}>0$ if and only if $\lambda  \geq  \mu$ in the dominance order, i.e., $\sum_{i=1}^s \lambda_i \geq  \sum_{i=1}^s \mu_i$ for every $s=1,\dots, r$. 
\item[(3)] $K_{\lambda, \mu}$ is the number of semi-standard tableaux (i.e., tableaux with weakly increasing rows and strictly increasing columns) of shape $\lambda$ with  entries $1,\dots,n$ and  multiplicities  given by $\mu$ (i.e. $\mu_1$ entries are equal to $1$, $\mu_2$ entries are equal to $2$, and so on). 
\end{itemize} 
\end{prop}

This allows us to reformulate Theorem \ref{multidegdet} in terms of Schur polynomials. 

\begin{thm} 
\label{schurmultidegdet}
The multidegree of the determinantal ring $S/I_t(X)$ of an $m\times n$ matrix of variables $X$ with $2\leq t\leq \min\{m,n\}$ and with respect to the $\ZZ^n$-graded structure induced by $\deg(x_{ij})=e_j$ is 
$$\Deg_{S/I_t(X)}(Z_1,\dots,Z_n)=(Z_1\cdots Z_n)^{t-2} s_\lambda(Z) $$
where  
$$\lambda=\ell^{(t-1)}=\underbrace{\ell,\ell,\dots,\ell}_ {(t-1)-\mbox{times}} \mbox{ where }  \ell=m+1-t.$$
\end{thm}  

Summing up, we have the following combinatorial description of the multidegrees of determinantal rings. 

\begin{thm} 
\label{multidegdet2}
Let $X$ be an $m\times n$ matrix of variables and consider $S=K[X]$ with the $\ZZ^n$-graded structure induced by $\deg(x_{ij})=e_j$.
Let $2\leq t\leq \min\{m,n\}$ and let $S/I_t(X)$ be the associated $\ZZ^n$-graded determinantal ring. Set $\lambda=\ell^{(t-1)}$ with  $\ell=m+1-t$. 
For $b\in \NN^n$ with $|b|=\dim S/I_t(X)-n$, the multidegrees  $e_{b}(S/I_t(X))$ satisfy the following properties: 
\begin{itemize}
\item[(1)] $e_{b}(S/I_t(X))$ is a symmetric function of  $b$.  
\item[(2)] $e_{b}(S/I_t(X))> 0$ if and only if $t-2\leq b_i\leq m-1$ for every $i=1,\dots, n$. 
\item[(3)] Set $c_i=b_i+1$ for $i=1,\dots, n$. Then $e_{b}(S/I_t(X))$ is the number of families of non-intersectiong paths from $p_1=(1,n), p_2=(2,n), \dots, p_{t-1}=(t-1,n)$  to $q_1=(m,1), q_2=(m,2), \dots , q_{t-1}=(m, t-1)$ with exactly   $c_i$  points on the $i$-th column for $i=1,\dots, n$. 
\item[(4)] Set $\mu_i=b_i-(t-2)$  for $i=1,\dots, n$. Then $e_{b}(S/I_t(X))$ equals the Kostka number $K_{\lambda, \mu}$, that is, the number of semi-standard tableaux  of shape $\lambda$ with  entries $1,\dots,n$ and  multiplicities  given by $\mu$. 
\end{itemize} 
\end{thm} 

\begin{proof} (1) The fact that  $e_{b}(S/I_t(X))$ is a symmetric function of  $b$ follows  from the fact that   $I_t(X)$ is invariant under the permutation of the columns or from Theorem \ref{multidegdet}. Furthermore (3) and (4) are reformulations of Theorem \ref{multidegdet} and (\ref{MD2}). Finally 
(2) follows by applying Proposition \ref{Kostka numbers} part (2).
\end{proof} 
 
\begin{ex}\label{443multideg} 
For $m=n=4$ and $t=3$ we have that $\ell=m+1-t=2$, $t-1=2$ and $\lambda=2,2$. With $Z=Z_1,\dots, Z_4$ and $z=\prod_{i=1}^4  Z_i$ we have 
\begin{eqnarray*}
\Deg_{R/I_3(X)}(Z)= z \ s_{2,2}(Z)=
z \left( m_{(2,2,0,0)}+m_{(2, 1, 1, 0)}+2 m_{(1,1,1,1)} \right)= 
\\ m_{(3,3,1,1)}+m_{(3, 2, 2, 1)}+2 m_{(2,2,2,2)}
\end{eqnarray*}

By Proposition \ref{Kostka numbers}, the coefficients appearing in the expression have two combinatorial interpretations. For example, the coefficient $2$ of $m_{(2,2,2,2)}$ is the Kostka number $K_{22, 1111}$, i.e., the number of semistandard  tableaux of shape $2,2$ with entries $1,\dots,4$ and multiplicities given by $(1,1,1,1)$: 
$$
\begin{array}{cc}
1 & 2 \\ 3 & 4 
\end{array}
\quad 
\begin{array}{cc}
1 & 3 \\ 2 & 4 
\end{array}$$
Moreover, it is also the number of families of non-intersecting paths from $p_1=(1,4), p_2=(2,4)$ to $q_1=(4,1), q_2=(4,2)$  with $3$ points on each column: 
$$\begin{array}{cccc}
- & 1 &  1  &  1  \\
1 & 1 & -  &  2  \\
1 & - & 2 &  2  \\
 1& 2  & 2  &  -  \\
\end{array}
\quad  
\begin{array}{cccc}
- & - &  1  &  1  \\
1 & 1 & 1  &  2  \\
1 & 2 & 2 &  2  \\
 1& 2  & -  &  -  \\
\end{array}$$
We have also a geometric interpretation of the coefficient of $m_{(2,2,2,2)}$: it is the number of points that one gets by intersecting the  variety of $4\times 4$ matrices of rank at most $2$, regarded as a multigraded subvariety of $(\PP^3)^4$, with $L_1\times\dots \times L_4$. Here each $L_i$ is a generic linear space of codimension $2$ of $\PP^3$. Interpreting the four columns of the matrix as points in $\PP^3$, the rank $2$ condition means that the four points belong to a line in $\PP^3$ that must intersect the four generic lines $L_1,\dots, L_4$. How many lines intersect four general lines in $\PP^3$? The answer is $2$ and this a classical instance of Schubert calculus, see \cite[3.4.1]{EH} for a modern exposition.
\end{ex} 

The computation in Example \ref{443multideg} appears also in \cite{vW2} and it can be easily generalised. 

\begin{ex}\label{nnn-1multideg} 
For $m=n$ and $t=n-1$  we have that $\ell=m+1-t=2$  and $\lambda=2^{(n-2)}$. With  $Z=Z_1,\dots, Z_n$ and $z=\prod_{i=1}^n  Z_i$ we have
\begin{eqnarray*}
\Deg_{R/I_{n-1}(X)}(Z)= z^{n-2} \ s_{2^{(n-2)}}(Z)=
z^{n-2} \left( m_{\mu_1}+m_{\mu_2}+2 m_{\mu_3} \right) 
\end{eqnarray*}
with 
$$\mu_i=\left\{ 
\begin{array}{l}
(2^{(n-4)},2,2,0,0) \mbox{ if } i=1,\\
(2^{(n-4)},2,1,1,0) \mbox{ if } i=2,\\
(2^{(n-4)},1,1,1,1) \mbox{ if } i=3.
\end{array} 
\right.$$
\end{ex} 
 

\begin{rmk} 
In \cite[Chapter 15]{MS}  the authors compute the multidegree for a large family of determinantal ideals, the Schubert determinantal ideals,  with respect to  the finer multigrading $\deg X_{ij}=(e_i,-f_j) \in \ZZ^m\oplus \ZZ^n$ with $\{e_1,\dots,e_m\}$ and $\{f_1,\dots,f_n\}$ being the canonical bases of  $\ZZ^m$  and $\ZZ^n$. The ideal $I_t(X)$ is a Schubert determinantal ideals. For them the authors observe in that the multidegree is given by a Schur polynomial \cite[15.39]{MS} which is (at least apparently) different from the Schur polynomial that we have identified. 
\end{rmk} 

We are ready to state the main consequence of Theorem \ref{multidegdet2}.

\begin{thm}\label{freemultideg}
Let $S/I_t(X)$ be the determinantal ring of the $m\times n$ matrix of variables $X$ with the $\ZZ^n$-graded structure induced by $\deg(x_{ij})=e_j$, with $2\leq t\leq\min\{m,n\}$. Then $S/I_t(X)$ has a multiplicity-free multidegree if and only if $t=2$ or $ t=\min\{m,n\}$.
\end{thm} 

\begin{proof} We need to prove two assertions:

\medskip  
{\bf Claim 1)} if  $t=2$ or $ t=\min\{m,n\}$, then $e_b(S/I_t(X))\in \{0, 1\}$ for all $b\in \ZZ^n$ with $|b|=\dim S/I_t(X)-n$. 

\medskip
{\bf Claim 2)} if $2<t<\min\{m,n\}$, then there exists $b\in \ZZ^n$ with $|b|=\dim S/I_t(X)-n$ such that  $e_b(S/I_t(X))>1$. 

\medskip 
Claim 1) for $t=2$ follows immediately from Theorem \ref{multidegdet}.  Claim 1) for $t=m\leq n$ follows from the description of $e_b(S/I_t(X))$ in Theorem \ref{multidegdet2}(4)  in terms of semi-standard tableau, since the corresponding shape is a single column. 
Finally Claim 1) for $t=n\leq m$ can be treated as follows. By Theorem \ref{multidegdet2}(1), $e_b(S/I_n(X))>0$ if and only if  $n-2\leq b_i \leq  m-1$ and $|b|=\dim S/I_n(X)-n$. Setting   $b_i=(m-1)-c_i$ and rewriting the conditions with respect to $(c_1,\dots, c_n)$, we have that $e_b(S/I_n(X))>0$ if and only if $(c_1,\dots, c_n)\in \NN^n$ and $\sum_{i=1}^n c_i=m-n+1$. Hence there are exactly $\binom{m}{n-1}$ elements $b\in \NN^n$ such that $e_b(S/I_n(X))>0$. By (\ref{sum-mm}), $\sum_b e_b(S/I_n(X))$ gives the ordinary multiplicity of $S/I_n(X)$, which is $\binom{m}{n-1}$. It follows that $e_b(S/I_n(X))=1$ whenever  $e_b(S/I_n(X))>0$. 

For Claim 2), by Theorem \ref{multidegdet2}(4) it suffices to show that there exists a $\mu\in \NN^n$ with $|\mu|=(t-1)\ell$, $\ell=m+1-t$, and such that there are at least two semi-standard tableaux of shape $\ell^{(t-1)}$ and entries $1,\dots, n$ with multiplicities given by $\mu$. One can take $$\mu_i=
\left\{ \begin{array}{ll} 
\ell  & \mbox{ if } i=1,2,\dots, t-3, \\
\ell-1  & \mbox{ if }i=t-2, t-1,  \\
1        & \mbox{ if } i=t, t+1,  \\
0     & \mbox{ if  }  i=t+2, \dots,  n.
\end{array}
\right.$$
For $i=1,\dots, t-3$, the $i$-th row of any semi-standard tableau of shape $\ell^{(t-1)}$ and multiplicities given by $\mu$ consists of exactly $\ell$ entries equal to $i$. So we may simply assume that $t=3$. Similarly we may assume that $n=t+1=4$, so that $\mu=(\ell-1,\ell-1,1,1)$. Now for $j=1,\dots, m-4$ the $j$-th column must have entries $1$ and $2$. Again we may then assume that $m=4$, hence $\ell=2$. Now it is clear that there are exactly two tableaux of shape $2,2$ and multiplicities given by $(1,1,1,1)$, namely those described in Example \ref{443multideg}.  
\end{proof} 

As a corollary of Theorem \ref{freemultideg} we have
\begin{cor} 
\label{CSdet} 
Let $S/I_t(X)$ be the determinantal ring  of the $m\times n$ matrix of variables $X$ with the $\ZZ^n$-graded structure induced by $\deg(x_{ij})=e_j$, with $1\leq t\leq \min\{m,n\}$. Then $I_t(X)$ is Cartwright-Sturmfels  if and only if $t=1,2$ or $t=\min\{m,n\}$.  
\end{cor} 

\begin{proof} The case $t=1$ is obvious, so we may assume $t>1$. 
Since $I_t(X)$ is a relevant prime, the conclusion follows combining Theorem \ref{freemultideg} and Proposition \ref{oldprop2}. 
\end{proof} 

The fact that $I_2(X)$ is Cartwright-Sturmfels has been proved directly (i.e. without using Proposition \ref{oldprop2}) by Cartwright and Sturmfels in \cite{CS}, hence the name. For $I_t(X)$ with $t=\min\{m,n\}$ it has been proved directly in \cite{CDG1}.  
  
Combining Corollary \ref{CSdet} with Proposition \ref{closures}, which states that any multigraded linear section of a Cartwright-Sturmfels ideal remains Cartwright-Sturmfels, we obtain the following result, originally proved in \cite[Main Theorem]{CDG2}. 

\begin{thm}\label{detA}
Let $A=(a_{ij})$ be an $m\times n$ matrix whose entries are $\ZZ^n$-multigraded  
with $\deg(a_{ij})=e_j$ for all $i,j$.
Let $I_t(A)$ be the ideal of $t$-minors of $A$.  Then $I_t(A)$ is Cartwright-Sturmfels for  $t=1,2,\min\{m,n\}$.  
\end{thm}  

In particular when $t=1,2,\min\{m,n\}$ then $I_t(A)$ has all the properties of the Cartwright-Sturmfels ideals listed in Proposition \ref{oldprop}.  When $m\leq n$ every maximal minor of $A$ has a different $\ZZ^n$-degree and we obtain a more precise statement.

\begin{cor} 
Under the assumptions of Theorem \ref{detA}, if $m\leq n$ then the maximal minors of $A$ form a universal Gr\"obner basis of $I_m(A)$.
\end{cor}  

\begin{rmk}
It is natural to ask whether the ideal $I_t(A)$ can be Cartwright-Sturmfels, under the assumptions of Theorem \ref{detA} and for $2<t<\min\{m,n\}$. The answer is yes if $A$ is very special (for example when $I_t(A)=0$) and no for a general enough $A$ (for example if $I_t(A)$ has the expected codimension). 
Nevertheless, notice that the generators of $I_t(A)$ have squarefree $\ZZ^n$-degrees, hence cannot have factors of multiplicity larger than one. This suggest that $I_t(A)$ might be radical. It turns out that this is not the case: In \cite[Example 7.2]{CW} the authors give examples of non-radical coordinate sections of determinantal ideals. 
\end{rmk} 


The multidegree of $S/I_t(A)$ for $t=2,\min\{m,n\}$ was essentially computed in \cite{CDG3}. Indeed, in that paper we computed the prime decomposition of the multigraded gin of $I_t(A)$, from which the multidegree is easily derived.

\section{Schubert determinantal ideals and matrix Schubert varieties}\label{sect:schubert}

Matrix Schubert varieties were introduced by Fulton in \cite{F}. They are defined by rank conditions. 
In this section, we show that many defining ideals of matrix Schubert varieties are Cartwright-Sturmfels. We start by fixing the notation and recalling the definitions.

Let $X=(x_{i,j})$ be an $n\times n$ matrix of variables over a field $K$ and let $S=K[X]=K[x_{ij} : 1\leq i,j\leq n]$. We consider the $\ZZ^n$-graded structure on $S$ induced by $\deg(x_{ij})=e_j\in \ZZ^n$. In the notation of Section \ref{multidegreeStructure}, this corresponds to letting $m_j=n-1$ for $j=1,\dots, n$. 
For $M$ a matrix of size $n\times n$ and $a,b\in\{1,\ldots,n\}$, let $M_{a\times b}$ be the submatrix of $M$ consisting of the entries in position $(i,j)$ where $i\leq a$ and $j\leq b$. 

Denote by $\mathcal{S}_n$ the group of permutations on the set $\{1,\ldots,n\}$ and let $\omega\in\mathcal{S}_n$. We write $\omega$ in line notation, i.e., $\omega=\omega_1\cdots\omega_n$ if $\omega(i)=\omega_i$. We associate to $\omega$ the rank function $r_\omega:\{1,\ldots,n\}^2\to\mathbb{N}$ defined by $$r_\omega(i,j)=|\{(k,\ell)\leq (i,j)\; |\; k=\omega_\ell\}|.$$
In other words, let $P_\omega$ be the permutation matrix corresponding to $\omega$, that is, $P_\omega e_j=e_{\omega_j}$. Then $r_\omega(i,j)$ is the number of ones in the submatrix $(M_\omega)_{i\times j}$.
Notice that this is the transpose of the usual definition of rank function, see e.g. \cite[Section 1.3]{KM}. We choose this notation in order to be coherent with the $\mathbb{Z}^n$-grading that we defined in Section \ref{multidegreeStructure}.

\begin{defn}
Let $m\leq n$ and let $\omega=\omega_1\cdots\omega_n\in\mathcal{S}_n$ and $\upsilon=\upsilon_1\cdots\upsilon_m\in\mathcal{S}_m$. We say that $\omega$ contains $\upsilon$ if there are 
$1\leq i_1<\ldots<i_m\leq n$ such that $\omega_{i_j}<\omega_{i_\ell}$ if and only if $\upsilon_j<\upsilon_\ell$.
Else, we say that $\omega$ avoids $\upsilon$.
A permutation $\omega$ is vexillary if it avoids the permutation $2143\in\mathcal{S}_4$.
\end{defn}

To each permutation, one associates a Rothe diagram and an essential set as follows.

\begin{defn}
The Rothe diagram associated to $\omega\in\mathcal{S}_n$ is
$$D_\omega=\{(i,j)\mid 1\leq i,j\leq n,\; \omega_j>i,\; (\omega^{-1})_i>j\}.$$
The essential set of $\omega$ is 
$$\ess(\omega)=\{(i,j)\in D_\omega\mid (i+1,j), (i,j+1)\not\in D_\omega\}.$$
\end{defn}
Notice that, as for the rank function, these are the transpose of the usual Rothe diagram and essential set of a permutation.

\begin{ex}\label{ex:1432}
Let $\omega=1432\in\mathcal{S}_4$. The permutation $1432$ is vexillary and has Rothe diagram $D_{1432}=\{(2,2),(2,3),(3,2)\}$ and essential set $\ess(1432)=\{(2,3),(3,2)\}$.
\end{ex}

\begin{defn}\label{defn:schubertvars}
Let $\omega\in\mathcal{S}_n$.
The Schubert determinantal ideal associated to $\omega$ is 
$$I_\omega=\sum_{i,j=1,\ldots,n}I_{r_\omega(i,j)+1}(X_{i\times j})\subseteq S.$$
The matrix Schubert variety associated to $\omega$ is the corresponding affine variety, i.e.
$$\mathcal{X}_\omega=\{M\in K^{n\times n}\mid {\rm rk}(M_{i\times j})\leq r_\omega(i,j) \mbox{ for all } 1\leq i,j\leq n\}.$$
\end{defn}

Notice that Schubert determinantal ideals are $\mathbb{Z}^n$-graded, since the minors that generate them are. Moreover, by \cite[Lemma 3.10]{F} we have that
$$I_\omega=\sum_{(i,j)\in\ess(\omega)}I_{r_\omega(i,j)+1}(X_{i\times j}).$$
Let \begin{equation}\label{eqn:Y}
Y_\omega=\cup_{(i,j)\in\ess(\omega)}X_{i\times j}
\end{equation} be the one-sided subladder of $X$ whose lower outside corners are the elements of the essential set of $\omega$. $Y_\omega$ is the set of variables of $X$ that appear in at least one of the generators of $I_\omega$. Consider the ideal generated in $K[Y_\omega]$ by the minors that generate $I_\omega$, that is, consider $I_\omega\cap K[Y_\omega]\subseteq K[Y_\omega]$. Then $I_\omega\cap K[Y_\omega]$ is $\mathbb{Z}^\nu$-graded, where $\nu=\max\{j\mid (i,j)\in\ess(\omega)\mbox{ for some } i\}$ is the number of columns of $Y_\omega$.

The family of Schubert determinantal ideals contains that of one-sided ladder determinantal ideals. More precisely, 
consider mixed one-sided ladder determinantal ideals. These are a generalization of the classical one-sided ladder determinantal ideals, where the ladder can have corners in the same row or column and we take minors of different sizes in different regions of the ladder, see e.g. \cite[Definition 1.4]{G07}. In \cite[Proposition 9.6]{F} it is shown that the family of mixed one-sided ladder determinantal ideals coincides with that of Schubert determinantal ideals associated to vexillary permutations. Every permutation is vexillary for $n\leq 3$ and the only non-vexillary permutation in $\mathcal{S}_4$ is $2143$. However, for large $n$, the proportion of vexillary permutations tends to zero as $n$ tends to infinity \cite{M91}. Therefore, for large enough $n$, (mixed) one-sided ladder determinantal ideals are a small subset of Schubert determinantal ideals.

\begin{ex}
Consider the permutation $\omega=1432\in\mathcal{S}_4$ from Example \ref{ex:1432}. Its Schubert determinantal ideal is $I_{1432}=I_2(X_{2\times 3})+I_2(X_{3\times 2})=I_2(Y_{1432})$ where $Y_{1432}=X_{2\times 3}\cup X_{3\times 2}$ is the subladder of $X_{3\times 3}$ consisting of its first two rows and columns. 

The ideal $I_{1432}\subseteq K[x_{ij}\mid 1\leq i,j\leq 4]$ is $\mathbb{Z}^4$-graded with respect to the grading induced by letting $\deg(x_{ij})=e_j\in\mathbb{Z}^4$. One can also regard $I_{1432}$ as an ideal in $K[Y_{1432}]=K[x_{ij}\mid 1\leq i,j\leq 3, (i,j)\neq (3,3)]$, which is $\mathbb{Z}^3$-graded graded with respect to the grading induced by letting $\deg(x_{ij})=e_j\in\mathbb{Z}^3$.
\end{ex}

The next result follows by combining a recent result by Fink, M\'esz\'aros, and St. Dizier \cite{FMS} with results by Knutson and Miller \cite{KM} and by Brion \cite{B}. It characterizes the Schubert determinantal ideals which are Cartwright-Sturmfels.

\begin{thm}\label{thm:schubert}
Assume that $K$ is algebraically closed.
Let $\omega\in\mathcal{S}_n$ and let $I_\omega\subseteq S$ be the associated Schubert determinantal ideal. The following are equivalent:
\begin{enumerate}
    \item $I_\omega\in CS(S)$,
    \item $I_\omega\cap K[Y_\omega] \in CS(K[Y_\omega])$,
    \item $\omega$ avoids the permutations 12543, 13254, 13524, 13542, 21543, 125364, 125634, 215364, 215634, 315264, 315624, and 315642.
\end{enumerate}
If this is the case, then the multigraded generic initial ideal and all the initial ideals of $I_\omega$ are Cohen-Macaulay. Moreover, $I_\omega$ has a universal Gr\"obner basis consisting of elements of multidegree $\leq (1,\ldots,1,0,\ldots,0)\in\mathbb{Z}^n$, where the number of ones appearing in the vector is equal to the number of columns of $Y_\omega$.
\end{thm}

\begin{proof}
By \cite[Theorem 4.8]{FMS}, the Schubert polynomial of $\omega$ is multiplicity-free if and only if $\omega$ avoids the permutations 12543, 13254, 13524, 13542, 21543, 125364, 125634, 215364, 215634, 315264, 315624, and 315642. Moreover, the Schubert polynomial of $\omega$ coincides with the dual multidegree of $S/I_\omega$ by \cite[Theorem A]{KM}. Therefore, $\omega$ avoids the 12 permutations listed above if and only if the only coefficients appearing in the dual multidegree of $S/I_\omega$ are zero and one. Notice moroever that $S/I_\omega$ and $K[Y_\omega]/I_\omega\cap K[Y_\omega]$ have the same dual multidegree.

If $I_\omega$ is Cartwright-Sturmfels, then $S/I_\omega$ and $K[Y_\omega]/I_\omega\cap K[Y_\omega]$ have multiplicity-free multidegree by Proposition \ref{oldprop2}. Moroever, the ideal $I_\omega\subseteq S$ is prime by \cite[Proposition 3.3]{F}. In \cite[Lemma 2.3]{CDG4}, we discussed the relation between the multidegree $\Deg_M(Z)$ and the dual multidegree $\Deg^*_M(Z)$. In particular we showed that, under our assumptions, they are two different encodings of the same numerical data. In particular,
$\Deg_M(Z)$ is multiplicity-free if and only if the only coefficients in $\Deg^*_M(Z)$ are zero and one. This proves that (1) implies (3) and (2) implies (3).

Conversely, suppose that (3) holds. Then $S/I_\omega$ and $K[Y_\omega]/I_\omega\cap K[Y_\omega]$ have multiplicity-free multidegrees by \cite[Lemma 2.3]{CDG4}. Since $I_\omega\subseteq S$ is prime, the multigraded generic initial ideal of $I_\omega$ is radical and Cohen-Macaulay by \cite[Theorem 1]{B} (see also \cite[Theorem 1.11]{CDG4} for a formulation in our terminology). This proves (1). The same argument proves (2).
The rest of the statement follows from Proposition \ref{oldprop}.
\end{proof}

The next result follows by combining Theorem \ref{thm:schubert} and Proposition \ref{closures}. 

\begin{cor}
Let $S=K[x_{ij}\mid 1\leq j\leq n,\ 0\leq i\leq m_j]$ be endowed with the standard $\ZZ^n$-grading induced by $\deg(x_{ij})=e_j\in\ZZ^n$ and assume that $K$ is algebraically closed. 
Let $A=(a_{ij})$ be an $n\times n$
matrix whose entries are $\ZZ^n$-multigraded with $\deg(a_{ij})=e_j\in\mathbb{Z}^n$ for all $i,j$.
Let $\omega\in\mathcal{S}_n$ and assume that $\omega$ avoids the permutations 12543, 13254, 13524, 13542, 21543, 125364, 125634, 215364, 215634, 315264, 315624, and 315642.
Let $$I_\omega(A)=\sum_{i,j=1,\ldots,n}I_{r_\omega(i,j)+1}(A_{i\times j})
=\sum_{(i,j)\in\ess(\omega)}I_{r_\omega(i,j)+1}(A_{i\times j})\subseteq S.$$ Then $I_\omega(A)$ is Cartwright-Sturmfels.
\end{cor}

In \cite{HPW} Hamaker, Pechenik, and Wiegand study the following system of generators for Schubert determinantal ideals.

\begin{defn}
Let $X^\prime$ be the matrix obtained from $X=(x_{ij})$ by specializing $x_{ij}$ to $0$ whenever $r_\omega(i,j)=0$. The CDG generators of $I_\omega$ are the elements of the set
$$\{x_{ij}\mid r_\omega(i,j)=0\}\cup\{(r_\omega(i,j)+1)-\mbox{minors of } X^\prime_{i\times j}\mid (i,j)\in\ess(\omega)\}.$$
\end{defn}

In their paper, Hamaker, Pechenik, and Weigandt formulate the following conjecture, which was later proved by Klein in \cite{K}.

\begin{thm}[Conjecture 7.1 in \cite{HPW}, Corollaries 3.17 and 4.2 in \cite{K}]\label{thm:CGDgens}
Let $\omega\in\mathcal{S}_n$. The CDG generators are a diagonal Gr\"obner basis for $I_\omega$ if and only if $\omega$ avoids the permutations 13254, 21543, 214635, 215364, 215634, 241635, 315264, and 4261735.
\end{thm}

Combining Theorem \ref{thm:schubert} and Theorem \ref{thm:CGDgens}, one obtains the following immediate corollary.

\begin{cor}
Let $\omega\in\mathcal{S}_n$. If $I_\omega$ is Cartwright-Sturmfels, then the CDG generators are a diagonal Gr\"obner basis.
\end{cor}

\begin{proof}
Notice that the permutations 214635 and 241635 contain 13524 and the permutation 4261735 contains 315624. Therefore, if $\omega$ avoids the list of permutations in the statement of Theorem \ref{thm:schubert}, then it also avoids the permutations listed in Theorem \ref{thm:CGDgens}.
\end{proof}

By comparing the lists of permutations in the statements of Theorem \ref{thm:schubert} and Theorem \ref{thm:CGDgens}, one sees immediately that there are Schubert determinantal ideals which are not Cartwright-Sturmfels, but whose CDG generators are a diagonal Gr\"obner basis. 

\begin{ex}
Let $\omega=214635\in\mathcal{S}_6$ and let $I_\omega=(x_{11})+I_3(X_{3\times 4})+I_4(X_{5\times 4})$ be the associated Schubert determinantal ideal. Since the generators of $I_\omega$ are minors of $Y_\omega=X_{5\times 4}$, we may replace $X$ by $X_{5\times 4}$ and let $S=K[X_{5\times 4}]=K[x_{ij}\mid 1\leq i\leq 5, 1\leq j\leq 4]$. In particular, $S$ and $I_\omega$ are $\mathbb{Z}^4$-graded by letting $\deg(x_{ij})=e_j\in\mathbb{Z}^4$.

The dual multidegree of $S/I_\omega$ is $Z_1^3 Z_2 + Z_1^2 Z_2^2 + Z_1^3 Z_3 + 2Z_1^2 Z_2 Z_3 + Z_1 Z_2^2 Z_3 + Z_1^2 Z_3^2 + Z_1 Z_2 Z_3^2 + Z_1^3 Z_4 + 2Z_1^2 Z_2 Z_4 + Z_1 Z_2^2 Z_4 + 2Z_1^2 Z_3 Z_4 + 2Z_1 Z_2 Z_3 Z_4 + Z_1 Z_3^2 Z_4 + Z_1^2 Z_4^2 + Z_1 Z_2 Z_4^2 + Z_1 Z_3 Z_4^2$, in particular it is not multiplicity-free, so $I_\omega$ is not Cartwright-Sturmfels.

The CGD generators of $I_\omega$ are $x_{1,1}$, the 3-minors of $X^\prime_{3\times 4}$, and the 4-minors of $X^\prime_{5\times 4}$,
where
$$X^\prime=\begin{pmatrix}
0 & x_{12} & x_{13} & x_{14} \\
x_{21} & x_{22} & x_{23} & x_{24} \\ 
x_{31} & x_{32} & x_{33} & x_{34} \\
x_{41} & x_{42} & x_{43} & x_{44} \\
x_{51} & x_{52} & x_{53} & x_{54}
\end{pmatrix}.$$
Fix the lexicographic order with $x_{ij}>x_{k\ell}$ if either $i<k$ or $i=k$ and $j<\ell$. This is a diagonal term order. 
One can check by direct computation that the CDG generators are a Gr\"obner basis of $I_\omega$, e.g. by checking that the ideal generated by their leading terms has the same $\mathbb{Z}$-graded Hilbert series as $I_\omega$.
\end{ex}

\section{Binomial edge ideals}
\label{BinCS}

The next theorem appeared first as \cite[Theorem 2.1]{CDG3}.  Giulia Gaggero \cite{GG} pointed out to us that the proof given in \cite{CDG3} contains a mistake. Indeed the equation 
$$u (x_2 F_{1n}- x_1 F_{2n})=ux_nF_{12}$$
that is used in \cite[p. 242]{CDG3} is not correct. The problem  comes from the fact that in the proof we treated the $F_{ij}$ as if they were the $2$-minors of the matrix $\phi(X)$ (notations as in \cite{CDG3}) but that it is true only up to a scalar that has been used to make them monic, hence the mistake. Here we present a correct and somehow simpler proof of \cite[Theorem 2.1]{CDG3}. 

Let us set up the notation. Let $G$ be a graph on the vertex set $\{1,\dots, n\}$ and let $X$ be the $2\times n$ matrix of variables
$$X=\left(
\begin{array}{cccc}
x_1 & x_2 & \cdots & x_n \\
y_1 & y_2 & \cdots & y_n
\end{array} \right).$$
Denote by $\Delta_{ij}$ the $2$-minor of $X$ corresponding to columns $i,j$, i.e.,
$\Delta_{ij}=x_i y_j-x_j y_i$. We consider the binomial edge ideal of $G$
$$J_G=(  \Delta_{ij} : \{i,j\} \mbox{ is an edge of } G )$$
of $S=K[x_1,\dots, x_n, y_1,\dots, y_n]$. Binomial edge ideals were introduced in \cite{HHHKR} and \cite{O}. 
We consider the $\ZZ^n$-graded structure on $S$ induced by letting $\deg(x_i)=\deg(y_i)=e_i\in\ZZ^n$. 
 
\begin{thm}[\cite{CDG3}, Theorem 2.1]
The multigraded generic initial ideal of $J_G$ is generated by the monomials $y_{a_1}\cdots y_{a_v}x_i x_j$, where $i, a_1,  \cdots,  a_v, j$ is a path in $G$. 
In particular $J_G$ is a Cartwright-Sturmfels ideal, therefore all the initial ideals of $J_G$ are radical and $\reg(J_G)\leq n$. 
\end{thm} 

\begin{proof} 
Consider any term order such that $x_i>y_i$ for all $i$. 
To compute the generic initial ideal, we first apply a multigraded upper triangular transformation $\phi$  to $J_G$, i.e., for every $i$ we have $\phi(x_i)=x_i$ and $\phi(y_i)=\alpha_ix_i+y_i$ with $\alpha_i\in K$.  We obtain the matrix 
$$\phi(X)=\left(
\begin{array}{cccc}
x_1 & x_2 & \cdots & x_n \\
\alpha_1x_1+y_1 & \alpha_2x_2+y_2 & \cdots & \alpha_nx_n+y_n 
\end{array}
\right)$$
whose $2$-minors are
$$\phi(\Delta_{ij})=  \left | 
\begin{array}{cc}
x_i & x_j   \\
\alpha_ix_i+y_i & \alpha_jx_j+y_j  
\end{array}
\right |= (\alpha_j-\alpha_i)x_ix_j+\Delta_{ij}.
$$
Assume that $\alpha_j\neq \alpha_i$ for $i\neq j$. We multiply  $\phi(\Delta_{ij})$ by the inverse of $\alpha_j-\alpha_i$ and obtain
$$F_{ij}=x_ix_j-\lambda_{ij}\Delta_{ij}$$
with 
$$\lambda_{ij}=(\alpha_i-\alpha_j)^{-1},$$
so that $F_{ij}$ is monic. 
For later reference, we observe the following: for indices $1\leq i<j<k\leq n$, consider the
$S$-polynomial $S(F_{ik}, F_{jk})$. Expanding   $S(F_{ik}, F_{jk})$  we have
$$S(F_{ik}, F_{jk})=x_jF_{ik}-x_iF_{jk}=-\lambda_{jk}y_jx_ix_k+\lambda_{ik}y_ix_jx_k+(\lambda_{jk}-\lambda_{ik})y_kx_ix_j.$$ 
Performing division with reminder by $F_{ik}, F_{jk}, F_{ij}$ we obtain
$$S(F_{ik}, F_{jk})=-\lambda_{jk}y_jF_{ik}+\lambda_{ik}y_iF_{jk}+
(\lambda_{jk}-\lambda_{ik})y_kF_{ij}+r.$$
The remainder $r$ is 
$$r=-\lambda_{jk}y_j\lambda_{ik}\Delta_{ik}+\lambda_{ik}y_i\lambda_{jk}\Delta_{jk}+(\lambda_{jk}-\lambda_{ik})y_k\lambda_{ij}\Delta_{ij},$$
that is,
$$r=\lambda_{jk}\lambda_{ik}(-y_j\Delta_{ik}+ y_i \Delta_{jk}) +(\lambda_{jk}-\lambda_{ik})y_k\lambda_{ij}\Delta_{ij}.$$
Using the syzygy among minors  
$$y_i\Delta_{jk}-y_j\Delta_{ik}+y_k\Delta_{ij}=0$$ 
we have 
$$r=\lambda_{jk}\lambda_{ik}(-y_k\Delta_{ij}) +(\lambda_{jk}-\lambda_{ik})y_k\lambda_{ij}\Delta_{ij}=
(-\lambda_{jk}\lambda_{ik} +\lambda_{jk}\lambda_{ij}-\lambda_{ik}\lambda_{ij})y_k\Delta_{ij}$$
and 
$$-\lambda_{jk}\lambda_{ik} +\lambda_{jk}\lambda_{ij}-\lambda_{ik}\lambda_{ij}=0,$$ which can be checked by direct computation. Hence $r=0$ and
\begin{equation}\label{eq*} S(F_{ik}, F_{jk})=-\lambda_{jk}y_jF_{ik}+\lambda_{ik}y_iF_{jk}+
(\lambda_{jk}-\lambda_{ik})y_kF_{ij}.\end{equation}

Now  we return to the ideal $J_G$ and its image under $\phi$:
$$\phi(J_G)=(  F_{ij} : \{i,j\} \mbox{ is an edge of } G ).$$
Set 
$$F=\{  y_{a}F_{ij} :   i, a_1,  \cdots,  a_v, j \mbox{  is a path in } G\},$$ 
where
$$y_{a}=y_{a_1}\cdots y_{a_v}.$$
It suffices to  prove  that $F$  is a Gr\"obner basis for $\phi(J_G)$, for every $\phi$ such that $\alpha_j\neq \alpha_i$ for $i\neq j$. We first observe that $F\subseteq\phi(J_G)$, i.e., $y_{a}F_{ij} \in \phi(J_G)$ for every path $i, a_1,\ldots,a_v,j$ in $G$. Since $F_{ij}$ and $\phi(\Delta_{ij})$ differ only by a non-zero scalar, we may as well prove that $y_{a}\phi(\Delta_{ij}) \in \phi(J_G)$ for every path $i,a_1,\ldots,a_v,j$ in $G$. 
This is proved easily by induction on $v$, the case $v=0$ being trivial, applying to the matrix $\phi(X)$ the relation 
$$(z_{1i}, z_{2i})\Delta_{jk}(Z)\subseteq(\Delta_{ij}(Z),\Delta_{ik}(Z))$$
that holds for every $2\times n$ matrix $Z=(z_{ij})$ and every triplet of column indices $i,j,k$. 
In order to prove that $F$ is a Gr\"obner basis, we take two elements $y_{a}F_{ij}$ and $y_{b}F_{hk}$ in $F$ and prove that the corresponding $S$-polynomial reduces to $0$ via $F$. Here $a=a_1,\ldots,a_v$ and $b=b_1\ldots, b_r$ and $i,a,j$ and $h,b,k$ are  paths  in $G$.   We distinguish three cases: 
 
Case 1). If $\{i,j\}=\{h,k\}$, we may assume $i=h$ and $j=k$. The corresponding $S$-polynomial is $0$. 
 
Case 2). If $\{i,j\}\cap \{h,k\}=\emptyset$. Let $u=\GCD(y_a, y_b)$. Then $y_aF_{ij}=u (y_a/u)F_{ij}$ and $y_bF_{hk}=u (y_b/u)F_{hk}$. Notice that $(y_a/u)F_{ij}$ and $(y_b/u)F_{hk}$ have coprime leading terms, hence they form a Gr\"obner basis. If a Gr\"obner basis is multiplied with a single polynomial, the resulting set of polynomials is still a Gr\"obner basis. Hence $\{y_aF_{ij}, y_bF_{hk}\}$ is a Gr\"obner basis and the $S$-polynomial of $y_aF_{ij}, y_bF_{hk}$ reduces to $0$ using only $y_aF_{ij}, y_bF_{hk}$. 
 
Case 3). If $|\{i,j\}\cap\{h,k\}|=1$. Up to permuting the columns of $X$, we may assume that $i=1$, $h=2$ and $j=k=n$. 
Let $u=\LCM(y_a,y_b)$. We have
$$S(y_aF_{1n}, y_bF_{2n})=uS(F_{1n}, F_{2n}).$$
Using~(\ref{eq*}) with $i=1$, $j=2$ and $k=n$, and multiplying both sides by $u$, we obtain
\begin{equation}\label{eq**}\quad S(y_aF_{1n}, y_bF_{2n})=-\lambda_{2n}y_2uF_{1n}+\lambda_{1n}y_1uF_{2n}+
(\lambda_{2n}-\lambda_{1n})y_nuF_{12}.\end{equation}
Since~(\ref{eq*}) is  a division with reminder $0$ of  $S(F_{1n}, F_{2n})$ with respect to $F_{1n}, F_{2n}, F_{12}$, we may conclude that~(\ref{eq**}) is a division with reminder $0$ of $S(y_aF_{1n}, y_bF_{2n})$ with respect to the set $F$, provided that $y_2uF_{1n}, y_1uF_{2n}$ and $y_nuF_{12}$ are multiples of elements of $F$.  
Clearly $y_2uF_{1n}$ is a monomial multiple of $y_aF_{1n}$ and $y_1uF_{2n}$ is a monomial multiple of $y_bF_{1n}$. So we are left with $y_nuF_{12}$. 
If $u$ is divisible by a monomial $y_d=y_{d_1}\cdots y_{d_t}$ such that $1,d_1,\dots,d_t,2$ is a path in $G$, then $y_nuF_{12}$ is a multiple of $y_dF_{12}\in F$. On the other hand, if $u$ is not divisible by a monomial $y_d=y_{d_1}\cdots, y_{d_t}$ such that $1,d_1,\dots,d_t,2$ is a path in $G$, then 
$$\{1,a_1,\dots, a_v\} \cap \{2,b_1,\dots, b_r\} =\emptyset \mbox{ and } u=y_ay_b.$$  
In this case, $1,a,n,b,2$ is a path from $1$ to $2$ in $G$, hence $y_nuF_{12}=y_ny_ay_bF_{12}\in F$. 

This concludes the proof that the set $F$ is a Gr\"obner basis. The rest of the statement now follows from Proposition \ref{oldprop}. 
\end{proof}

\section{Multigraded closures of linear spaces} 
\label{ClosureCS}

We now return to the notation of Section \ref{multidegreeStructure}, in particular we let
$S=K[x_{ij}\mid 1\leq j\leq n,\ 0\leq i\leq m_j]$ with the standard $\ZZ^n$-grading induced by $\deg(x_{ij})=e_j$. 

Let $T=K[x_{ij}\mid 1\leq j\leq n,\ 1\leq i \leq m_j]\subseteq S$. 
Given a non-zero polynomial $f\in T$ we use the variables $x_{01},x_{02},\dots, x_{0n}$ 
to transform $f$ into a polynomial of $S$ which is $\ZZ^n$-graded in a ``minimal" way. 
Explicitly, let $f=\sum_{i=1}^r \lambda_i  w_i \in T\setminus 0$ where  $\lambda_i \in K\setminus\{0\}$ and $w_i$ is a monomial of degree $b_i=(b_{i1}, \dots, b_{in}) \in \ZZ^n$. Let $d=(d_1,\dotsm d_n)$ with $d_j=\max\{ b_{1j}, \dots, b_{rj} \}$. Then  the $\ZZ^n$-homogenization $f^{\hom}\in S$ of $f$ is defined as
$$f^{\hom}=\sum_{i=1}^r  \lambda_i   \left( \prod_{j=1}^n x_{0j}^{d_j-b_{ij}} \right) w_i.$$ 
Notice that $f^{\hom}$ is $\ZZ^n$-homogeneous of degree $d\in \ZZ^n$. 

Given an ideal $I\subseteq T$, its multigraded homogenization is the $\ZZ^n$-graded ideal of $S$ 
$$I^{\hom}=( f^{\hom} : f\in I\setminus 0 )\subseteq S.$$
Geometrically $I^{\hom}$ corresponds to the closure in $\PP^{(m_1,\dots, m_n)}$ of the affine variety defined by $I$. 
We denote by $I^\star$ the largest $\ZZ^n$-graded ideal of $T$ contained in $I$, i.e., the ideal generated by the $\ZZ^n$-graded elements of $I$. 

\begin{thm}[\cite{CDG3}, Theorem 3.1]\label{thm:homog}
Let $J$ be an ideal of $T$ generated by homogeneous polynomials of degree $1$ with respect to the $\ZZ$-graded structure. Then $J^{\hom}$ and $J^\star$ are Cartwright-Sturmfels ideals.   
\end{thm} 

\begin{rmk} Theorem \ref{thm:homog} was inspired by work of Ardila and Boocher. In their paper \cite{AB}, they consider the situation $m_1=\ldots=m_n=1$. Our result recovers and generalises some of their results. Indeed the case treated by Ardila and Boocher is special, in the sense that the ideal $J^{\hom}$ is not only a Cartwright-Sturmfels ideal but also  Cartwright-Sturmfels$^*$, a dual notion that is discussed in \cite{CDG4}. One important consequence of this fact is that the multigraded Betti numbers of $J$ equal the multigraded Betti numbers of any $\ZZ^n$-graded ideal with the same multigraded Hilbert function as $J$. In addition, any minimal multigraded system of generators is a universal Gr\"obner basis of $J$.
\end{rmk} 

\begin{ex} 
Let $n=3$ and $m_1=m_2=m_3=4$. We consider $J=(x_{i1}+x_{i2}+x_{i3} : i=1,2,3)$. 
With 
$$X=\left(\begin{array}{ccc}
x_{11} & x_{12} & x_{13} \\
x_{21} & x_{22} & x_{23} \\
x_{31} & x_{32} & x_{33} \\
x_{41} & x_{42} & x_{43} 
\end{array}\right)$$ 
we observe that 
$$X \left(\begin{array}{ccc}
1 \\
1 \\
1
\end{array}\right)
=0 \mod J,$$  
hence $I_3(X)\subseteq J$. Since $I_3(X)$ is $\ZZ^3$-graded we have also 
$I_3(X) \subseteq J^\star$. It turns out that actually one has $I_3(X)=J^\star$. This example can be generalised, see \cite[Example 5.2.]{CDG3} where the result is presented with the transposed graded convention, i.e.  with respect to the graded structure induced by $\deg(x_{ij})=e_i$. Summing up, one has that for every $m\geq n$ and $X=(x_{ij})$ matrix of variables with $\deg(x_{ij})=e_j$, the ideal $I_n(X)$ is equal to $J^\star$ where $J=(\sum_{j} x_{ij} : i=1,\dots, n)$. 
\end{ex} 

The ideals generated by linear forms are the only $\ZZ$-graded Cartwright-Sturmfels ideals. Hence Theorem \ref{thm:homog} could be a special instance of a more general fact, that we formulate as a question. 

\begin{quest} 
Let $I$ be a Cartwright-Sturmfels $\ZZ^n$-graded ideal of $S$. Suppose that we introduce a finer graded structure on $S$, say a $\ZZ^r$-graded structure with $r>n$  such that if two variables have the same  $\ZZ^r$-degree then they have the same $\ZZ^n$-degree. Then $I$ is not  necessarily  $\ZZ^r$-graded 
 and we may consider its $\ZZ^r$-homogenization $I^{\hom}\subseteq S[y_1,\dots, y_r]$ and homogeneous $\ZZ^r$-part $I^*$. Are $I^{\hom}$  and   $I^*$ Cartwright-Sturmfels ideals? 
 \end{quest} 



\begin{thebibliography}{99}

\bibitem{ACD} A. Aramova, K. Crona, E. De Negri, 
{\em Bigeneric initial ideals, diagonal subalgebras and bigraded Hilbert functions.} 
J. Pure Appl. Algebra 150 (2000), no. 3, 215--235. 

\bibitem{AB}  F. Ardila,  A. Boocher, 
{\em The closure of a linear space in a product of lines,}
J. Algebraic Combin. 43 (2016), no. 1, 199--235. 

\bibitem{BZ} D. Bernstein, A. Zelevinsky,
{\em Combinatorics of maximal minors. } 
J. Algebraic Combin. 2 (1993), no. 2, 111--121. 

\bibitem{B}  M. Brion, 
{\em Multiplicity-free subvarieties of flag varieties,} 
Contemp. Math. 331 (2003), 13--23 .

\bibitem{BC}  W. Bruns, A. Conca, 
{\em Gr\"obner bases and determinantal ideals,}  
Commutative algebra, singularities and computer algebra (Sinaia, 2002), 9--66, 
NATO Sci. Ser. II Math. Phys. Chem., 115, Kluwer Acad. Publ., Dordrecht, 2003. 

\bibitem{CS} D. Cartwright, B. Sturmfels,
{\em The {H}ilbert scheme of the diagonal in a product of  projective spaces},
Int. Math. Res. Not. 9 (2010), 1741--1771.

\bibitem{CR} Y. Cid-Ruiz, 
{\em Mixed multiplicities and projective degrees of rational maps}, 
J. Algebra 566 (2021), 136--162.

\bibitem{CDG1} A. Conca, E. De Negri, E. Gorla,
{\em Universal Gr\"obner bases for maximal minors},
Int. Math. Res. Not. 11 (2015), 3245--3262.

\bibitem{CDG3} A. Conca, E. De Negri, E. Gorla,
{\em Multigraded generic initial ideals of determinantal ideals},
Homological and Computational Methods in Commutative Algebra
A. Conca, J. Gubeladze, and T. Roemer Eds., Springer (2018). 

\bibitem{CDG4} A. Conca, E. De Negri, E. Gorla,
{\em Cartwright-Sturmfels ideals associated to graphs and linear spaces},
J. Comb. Algebra 2, no. 3 (2018), 231--257.

\bibitem{CDG2} A. Conca, E. De Negri, E. Gorla,
{\em Universal Gr\"obner bases and Cartwright-Sturmfels ideals},
Int. Math. Res. Not. 7 (2020), 1979--1991.

\bibitem{CW} A. Conca, V. Welker,
{\em Lov\'asz-Saks-Schrijver ideals and coordinate sections of determinantal varieties},
 Algebra Number Theory 13 (2019), no. 2, 455--484.  

\bibitem{CCLMZ} F. Castillo, Y. Cid-Ruiz, B. Li, J. Monta\~{n}o, N. Zhang,
{\em When are multidegrees positive?},
Adv. Math., 374, (2020), 107382, 34

\bibitem{E} D. Eisenbud,
{\em Commutative algebra. With a view toward algebraic geometry.}
Graduate Texts in Mathematics, 150. Springer-Verlag, New York (1995). 

\bibitem{EH} D. Eisenbud, J. Harris,
{\em 3264 and all that---a second course in algebraic geometry},
Cambridge University Press, Cambridge, 2016, xiv+616.
 
\bibitem{FMS} A. Fink, K. M\'esz\'aros, A. St. Dizier, 
{\em Zero-one Schubert polynomials},
preprint (2019), arXiv:1903.10332. 
 
\bibitem{F} W. Fulton, 
{\em Flags, Schubert polynomials, degeneracy loci, and determinantal formulas},
Duke Math. J. 65 (1992), no. 3, 381--420. 

\bibitem{GV} I. Gessel, G. Viennot, 
{\em Binomial determinants, paths, and hook length formulae,}
Adv. Math. 58 (1985), no. 3, 300--321. 

\bibitem{GG} G. Gaggero, {private communication}.

\bibitem{G07} E. Gorla,
{\em Mixed ladder determinantal varieties from two-sided ladders},
J. Pure Appl. Algebra 211 (2007), no. 2, 433--444.
 
\bibitem{HPW} 
Z. Hamaker, O. Pechenik, A. Weigandt,
{ \em Gr\"obner geometry of Schubert polynomials through ice},
preprint (2020), arXiv:2003.13719.

\bibitem{HT} J. Herzog, N.V. Trung, 
{\em Gr\"obner bases and multiplicity of determinantal and Pfaffian ideals},
Adv. Math. 96 (1992), no. 1, 1--37. 

\bibitem{HHHKR} J. Herzog, T. Hibi, F. Hreinsdottir, T. Kahle, J. Rauh, 
{\em Binomial edge ideals and conditional independence statements},
Adv. Appl. Math. 45 (2010), 317--333.  

\bibitem{K} P. Klein,
{\em Diagonal degenerations of matrix Schubert varieties},
preprint (2020), arXiv:2008.01717.

\bibitem{KM} A. Knutson, E. Miller,
{\em Gr\"obner geometry of Schubert polynomials},
Ann. Math., 161 (2005), 1245--1318.

\bibitem{M91} I.G. Macdonald,
{\em Notes on Schubert Polynomials}, D\'epartement de math\'ematiques et d'informatique, Universit\'e du Qu\'ebec, Montr\'eal (1991).

\bibitem{M}  I.G. Macdonald,
{\em Symmetric functions and Hall polynomials},
Oxford Classic Texts in the Physical Sciences,
The Clarendon Press, Oxford University Press, New York, 2015, xii+475.

\bibitem{MS}  E. Miller, B. Sturmfels,
{\em Combinatorial commutative algebra},
Graduate Texts in Mathematics 227,
Springer-Verlag, New York, 2005, xiv+417.

\bibitem{O} M. Ohtani, 
{\em Graphs and ideals generated by some 2-minors}, 
Comm. Alg. 39 (2011), 905--917.  

\bibitem{Sta} R.P. Stanley,
{\em Enumerative combinatorics. Vol. 2},
Cambridge Studies in Advanced Mathematics 62,
Cambridge University Press, Cambridge, 1999, xii+581.

\bibitem{S}  B. Sturmfels,  
{\em  Gr\"obner bases and Stanley decompositions of determinantal rings},
Math. Z. 205 (1990), no. 1, 137--144. 

\bibitem{SZ} B. Sturmfels, A. Zelevinsky,
{\em Maximal minors and their leading terms}, 
Adv. Math. 98 (1993), no. 1, 65--112. 
 
\bibitem{vW2}  B. L. van der Waerden,
{\em On varieties in multiple-projective spaces}, Nederl. Akad. Wetensch. Indag. Math. 40, (1978), no. 2, 303--312.

\end{thebibliography}
\end{document}